\documentclass[10pt]{amsart}

\usepackage{amsfonts,amssymb}
\usepackage[all]{xy}

\newtheorem{thm}{Theorem}[section]
\newtheorem{prop}[thm]{Proposition}
\newtheorem{cor}[thm]{Corollary}

\theoremstyle{remark}
\newtheorem{remark}{Remark}[section]
 
\theoremstyle{definition}
\newtheorem{definition}[thm]{Definition}

\def\qq{\mathbb{Q}}

\def\rr{\mathbb{R}}
\def\zz{\mathbb{Z}}
\def\CC{\mathbb{C}}
\def\SS{\mathbb{S}}
\def\mm{\mathcal{M}}
\def\aa{\mathcal{A}}

\def\QQ{\mathcal{Q}}
\def\cc{\mathcal{C}}
\def\jj{\mathcal{J}}
\def\tt{\mathcal{T}}
\def\vv{\mathcal{V}}
\def\hh{\mathcal{H}}
\def\HH{\mathbb{H}}
\def\deldelbar{\partial \overline{\partial}}
\def\gm{\mathbb{G}_\mathrm{m}}

\numberwithin{equation}{section}

\begin{document}

\title[Second variation of Zhang's $\lambda$-invariant]{Second variation of Zhang's $\lambda$-invariant on the moduli space of curves}

\author{Robin de Jong}



\begin{abstract} We compute the second variation of the $\lambda$-invariant, recently introduced by S. Zhang, on the complex moduli space $\mathcal{M}_g$  of curves of genus $g \geq 2$, using work of N. Kawazumi. As a result we prove that $(8g+4)\lambda$  is equal, up to a constant, to the $\beta$-invariant introduced some time ago by R. Hain and D. Reed. We deduce some consequences; for example we calculate the $\lambda$-invariant for each hyperelliptic curve, expressing it in terms of the Petersson norm of the discriminant modular form.
\end{abstract}

\maketitle

\thispagestyle{empty}

\section{Introduction}

Recently, independently S. Zhang \cite{zh} and N. Kawazumi \cite{kaw} introduced a new interesting real-valued function $\varphi$ on the moduli space $\mm_g$ of complex curves of genus $g \geq 2$. Its value at a curve $[X] \in \mm_g$ is
given as follows. Let $\mathrm{H}^0(X,\omega_X)$ be the space of holomorphic differentials on $X$, equipped with the hermitian inner product
\begin{equation} \label{defhodgemetric} (\alpha,\beta) \mapsto \frac{i}{2} \int_X \alpha \wedge \overline{\beta} \, . 
\end{equation}
Choose an orthonormal basis $(\eta_1,\ldots,\eta_g)$ of $\mathrm{H}^0(X,\omega_X)$, and put:
\[ \mu_X = \frac{i}{2g} \sum_{k=1}^g \eta_k \wedge \overline{\eta}_k \, . \]
Note that $\mu_X$ is a volume form on $X$, independent of the chosen basis; in fact,
$\mu_X$ can be identified with the pullback, along any Abel-Jacobi map, of a translation invariant $(1,1)$-form on the jacobian of $X$. 
Let $\Delta_\mathrm{Ar}$ be the Laplacian on $L^2(X,\mu_X)$ determined by setting
\[ \frac{ \partial \overline{\partial}}{\pi i} f = \Delta_\mathrm{Ar}(f)
\cdot \mu_X \, , \]
and let $(\phi_\ell)_{\ell=0}^\infty$ be an orthonormal basis of real eigenfunctions of $\Delta_\mathrm{Ar}$, with eigenvalues
$0=\lambda_0<\lambda_1 \leq \lambda_2 \leq \ldots$. 
\begin{definition}   \emph{(S. Zhang, N. Kawazumi)} We define the $\varphi$-invariant $\varphi(X)$ of $X$ to be the real number
\[ \varphi(X) = \sum_{\ell>0} \frac{2}{\lambda_\ell} \sum_{m,n=1}^g \left|
\int_X \phi_\ell \cdot \eta_m \wedge \overline{\eta}_n \right|^2 \, . \]
\end{definition}
It is not difficult to check that $\varphi(X)$ is indeed an invariant of $X$. 

One important reason for studying $\varphi$ is its significance in number theory, discovered by Zhang. Briefly, with it it can be shown that the Bogomolov conjecture (for curves over number fields) follows naturally from a standard conjecture of Hodge index type of Gillet-Soul\'e (we briefly recall this relationship in Section \ref{context} below). In order to see this implication, one uses that the $\varphi$-invariant is strictly \emph{positive}. Indeed, the $\varphi$-invariant can only vanish if each $\eta_m \wedge \overline{\eta}_n$ is perpendicular to all $\phi_\ell$. This would imply that each $\eta_m \wedge \overline{\eta}_n$ is proportional to $\mu_X$, but that is not the case under our assumption that $g \geq 2$ (cf. \cite{zh}, Remark after Proposition 2.5.3).

In view of its ramifications in number theory, it is of interest to try to study further properties of $\varphi$ in detail. A first important step is in the work \cite{kaw} of Kawazumi. The main theorem in \cite{kaw} furnishes an expression for the second variation of $\varphi$ on $\mm_g$, connecting $\varphi$ with certain canonical
$2$-forms over the universal curve $\mathcal{C}_g$ over $\mathcal{M}_g$ associated (following work of S. Morita) to the standard representation $H$ of $\mathrm{Sp}_{2g}(\zz)$, its third exterior power $\wedge^3 H$, and the `primitive part' $\wedge^3H/H$ of the latter. Here $H$ is seen as a subrepresentation of $\wedge^3H$ by wedging with the standard polarization form in $\wedge^2 H$.

In this paper we will use Kawazumi's result to establish some new properties of $\varphi$. More precisely: we determine its behavior in a neighbourhood of the boundary of $\mm_g$ in the Deligne-Mumford compactification $\overline{\mm}_g$, and we calculate $\varphi$ for hyperelliptic curves. 

In order to establish these results, it turns out to be convenient to consider the following variant of $\varphi$, also introduced in \cite{zh}.
\begin{definition} \emph{(Zhang)} Let $\delta_F$ on $\mm_g$ be Faltings's delta-invariant from \cite{fa}, and put $\delta=\delta_F - 4g \log(2\pi)$. We define the $\lambda$-invariant to be the real-valued function
\begin{equation} \label{lambdaphidelta} \lambda = \frac{g-1}{6(2g+1)}\varphi + \frac{1}{12}\delta  
\end{equation}
on $\mm_g$.
\end{definition}
Our main result in this paper is that $(8g+4)\lambda$ can be directly related to the function $\beta$ on $\mm_g$, introduced by R. Hain and D. Reed around ten years ago \cite{hrar}. This $\beta$-invariant is defined as follows.
Let $\mathcal{J}(\wedge^3H/H)$ be the Griffiths intermediate jacobian
fibration over $\mathcal{M}_g$ associated to $\wedge^3H/H$, and
let $\hat{\mathcal{B}}$ be the pullback, along the graph of the canonical polarization, of the standard $\gm$-biextension line bundle on
$\mathcal{J}(\wedge^3H/H) \times \check{\mathcal{J}(\wedge^3H/H)}$. 
The holomorphic line bundle $\hat{\mathcal{B}}$ comes with a natural hermitian metric
$\|\cdot\|_{\hat{\mathcal{B}}}$.
Let $\nu \colon \mathcal{M}_g \to \mathcal{J}(\wedge^3H/H)$ be the normal
function that maps each curve $X$ to the point in the intermediate
jacobian of $\wedge^3 H_1(X)/H_1(X)$ associated, by the Griffiths
Abel-Jacobi map, to the Ceresa cycle $X-X^-$ in
the jacobian of~$X$.

By a result of Morita one has $\nu^*\hat{\mathcal{B}} \cong
\mathcal{L}^{\otimes 8g+4}$,
where $\mathcal{L}=\det \mathrm{R}\pi_* \omega$ is the determinant of the Hodge
bundle on $\mathcal{M}_g$. The isomorphism is unique up to a
non-zero scalar, as the only invertible holomorphic functions on $\mm_g$ are scalars.
Denote by $\|\cdot \|_\mathrm{biext}$ a metric 
on $\mathcal{L}^{\otimes 8g+4}$ that one obtains by pulling back
$\|\cdot \|_{\hat{\mathcal{B}}}$ along $\nu$,
and transporting it to $\mathcal{L}^{\otimes 8g+4}$ using a Morita isomorphism.
Denote by $\| \cdot \|_{\mathrm{Hdg}}$ the metric on $\mathcal{L}^{\otimes 8g+4}$ induced by the Hodge metric (\ref{defhodgemetric}) on $\mathcal{L}$.

The Hain-Reed $\beta$-invariant \cite{hrar} is given by the ratio of these two metrics.
\begin{definition} \emph{(R. Hain, D. Reed)} We define the $\beta$-invariant on $\mm_g$ to be the real-valued function
\[ \beta = \log \left( \frac{ \| \cdot \|_\mathrm{biext}}{\| \cdot
\|_{\mathrm{Hdg}}} \right) \, . \]
\end{definition}

Note that the $\beta$-invariant is only defined up to an additive constant on
$\mathcal{M}_g$. 

Our main result is
\begin{thm} \label{realmain}
The equality $(8g+4)\lambda = \beta$ holds, up to a constant depending only on~$g$.
\end{thm}
The proof essentially boils down to a comparison of the second variations on $\mm_g$ of left and right hand side. Let
$\omega_{\mathrm{HR}}$ be the first Chern form of
$(\mathcal{L},\|\cdot\|_\mathrm{biext}^{1/(8g+4)})$ and let
$\omega_{\mathrm{Hdg}}$ be that of
$(\mathcal{L},\|\cdot\|_\mathrm{Hdg}^{1/(8g+4)})$. Then the differential equation
\[ \frac{\partial \overline{\partial} }{\pi i} \beta = (8g+4)
(\omega_\mathrm{HR} - \omega_\mathrm{Hdg})  \]
holds on $\mm_g$. With this, Theorem \ref{realmain} will follow from
\begin{thm} \label{main} The second variation of Zhang's $\lambda$-invariant over $\mathcal{M}_g$ satisfies
\[ \frac{\partial \overline{\partial} }{\pi i} \lambda =
\omega_\mathrm{HR} - \omega_\mathrm{Hdg} \, . \]
\end{thm}
Our proof of Theorem \ref{main} will be based on Kawazumi's calculation of the second variation of $\varphi$ referred to above. Thus, by equation (\ref{lambdaphidelta}), our contribution is essentially to find a convenient expression for the second variation of the Faltings delta-invariant over $\mm_g$.

Theorem \ref{realmain} allows us to determine the asymptotic behavior of $\varphi$ along the boundary components of the Deligne-Mumford compactification $\overline{\mathcal{M}}_g$. This now follows immediately from earlier results: in \cite{hrar} the asymptotic behavior of $\beta$ is computed, and for the Faltings delta-invariant $\delta_F$, this was done by J. Jorgenson \cite{jo} and R. Wentworth \cite{we}, independently. Combining these results using equation (\ref{lambdaphidelta}) we obtain
\begin{cor} \label{cor} Let $X \to D$ be a proper family of stable curves of genus $g \geq 2$
over the unit disk. Assume that $X$ is smooth and that $X_t$ is smooth for $t \neq 0$.
\begin{itemize}
\item If $X_0$ is irreducible with only one node, then
\[ \varphi(X_t) \sim -\frac{g-1}{6g} \log |t|  \]
as $t \to 0$.
\item If $X_0$ is reducible with one node and its components have genera $i$ and $g-i$ then
\[ \varphi(X_t) \sim -\frac{2i(g-i)}{g} \log |t| \]
as $t \to 0$.
\end{itemize}
\end{cor}
Here, if $f,g$ are two functions on the punctured unit disk, the notation $f \sim g$ denotes that $f-g$ is bounded as $t\to 0$.  The corollary implies that $\varphi$ is a `Weil function' on $\overline{\mm}_g$ (see \cite{jo}, Section 6 for a discussion). It would be very interesting to know whether $\varphi$ is also a Morse function on $\mm_g$, and if so, whether its behavior at its critical points can be effectively analyzed.

Our next result concerns the calculation of the $\lambda$- and $\varphi$-invariant of a hyperelliptic curve. Over the hyperelliptic locus in genus $g \geq 2$ one can make the metric $\|\cdot\|_{\mathrm{Hdg}}$ fairly explicit, using the discriminant modular form $\Delta_g$ (see Section \ref{hyperelliptic} for details). The metric $\|\cdot \|_{\mathrm{HR}}$ turns out to be constant over the hyperelliptic locus. Putting these facts together one is led to the following theorem.
\begin{thm} \label{hyp} Let $\|\Delta_g\|$ be the Petersson norm of $\Delta_g$ and let $n={2g \choose g+1}$. 
Then on the hyperelliptic locus in genus $g$, the $\lambda$-invariant is
given, up to a constant depending only on $g$, by
\[ (8g+4) n \, \lambda = -(8g+4)ng \log(2\pi) - g \log \|\Delta_g\| \, . \]
\end{thm}
By using a recent result due to K. Yamaki \cite{ya} we will prove that the constant implied by the theorem actually vanishes. 
\begin{cor} On the hyperelliptic locus in genus $g$, the $\varphi$-invariant is
given by
\[ (2g-2)n \, \varphi  = -8(2g+1)ng \log(2\pi) - 3g \log
\|\Delta_g\|  - (2g+1)n \, \delta_F  \, , \]
where $\delta_F$ is Faltings's delta-invariant.
\end{cor}
The result of Yamaki and the above corollary together confirm a conjecture about the value of $\varphi$ for hyperelliptic curves put forward in \cite{djsymm}.

\section{Number theoretic context} \label{context}

Before we start, we would like to explain briefly the role played by $\varphi$ in number theory. We refer to \cite{zh} for a more detailed exposition and proofs. This section can be read independently of the others.
 
Let $k$ be a number field, and let $X$ be a smooth,
projective and geometrically connected curve of genus $g \geq 2$ with semistable reduction over $k$. Then the function $\varphi$ gives rise to real invariants associated to each archimedean place $v$ of $k$, by considering the base change of $X$ along $v$. As is explained in \cite{zh}, one also has a $\varphi$-invariant   associated to each non-archimedean place of $k$. In this case, the definition of 
$\varphi$ is in terms of the combinatorics of the semistable reduction graph of $X$ at $v$. This `finite' $\varphi$-invariant vanishes at places $v$ of good reduction.

Now let $\xi$ be a $k$-rational point of
$\mathrm{Pic}^1 X$ such that $(2g-2)\xi$ is the class of a canonical divisor on $X$.
Let $\Delta_\xi$ in $\mathrm{CH}^2(X^3)$ be the modified diagonal
cycle in $X^3$ associated to $\xi$ as defined by B. Gross and C. Schoen in \cite{gs}. We call $\Delta_\xi$ a canonical Gross-Schoen cycle on $X^3$.
It turns out that the invariant $\varphi$ occurs as a local
contribution in a formula relating the self-intersection $\langle \Delta_\xi,
\Delta_\xi \rangle$ (defined in \cite{gs})
of $\Delta_\xi$ to the admissible self-intersection of
the relative dualizing sheaf $(\omega, \omega)_a$ (defined in \cite{zhadm}) of $X$. More precisely, we have that the formula
\begin{equation} \label{omegasquared} (\omega,\omega)_a = \frac{2g-2}{2g+1} \left( \langle \Delta_\xi,
\Delta_\xi \rangle + \sum_v \varphi(X_v) \log Nv \right)  
\end{equation}
holds. Here the sum is taken over all places $v$ of $k$. The
$Nv$ are certain canonical local factors, and we have written $X_v$ for $X \otimes k_v$. This formula is in fact the main result of \cite{zh}.

The significance of formula (\ref{omegasquared}) is that it sheds a new light on the strict positivity of $(\omega,\omega)_a$ (ex-Bogomolov conjecture, proved in the nineties by E. Ullmo \cite{ul} and Zhang \cite{zhbog}). First of all, the self-intersection of the canonical Gross-Schoen cycle
$\langle \Delta_\xi,\Delta_\xi \rangle$ should be non-negative by a standard conjecture (of Hodge index type) of Gillet-Soul\'e (cf. \cite{zh}, Section~2.4).
Next, if $v$ is non-archimedean, the invariant $\varphi(X_v)$ is non-negative. This follows from a result of Z. Cinkir \cite{ci}. Finally, for archimedean places $v$, the value $\varphi(X_v)$ is \emph{positive}, as explained in the Introduction. These remarks together show that Gillet-Soul\'e's standard conjecture naturally implies the strict positivity of $(\omega,\omega)_a$, via equation (\ref{omegasquared}).

Actually, Cinkir in \cite{ci}, Theorem 2.11 proves the following conjecture of Zhang from \cite{zh}: for the non-archimedean $\varphi$-invariant there exists a lower bound
\begin{equation} \label{Zhangconj}
\varphi(X_v) \geq c(g) \, \delta_0 + \sum_{i=1}^{ [g/2]} \frac{2i(g-i)}{g}
\delta_i
\end{equation}
where for each $i=0,\ldots,[g/2]$ the invariant $\delta_i$ denotes the number of
singular points in the special fiber of $X_v$ such that the local normalization
of that fiber at $x$ is connected if $i=0$ or a disjoint union of two curves of
genera $i$ and $g-i$ if $i>0$, and where $c(g)$ is a positive constant depending only
on $g$. In fact, one can take $c(g)=\frac{g-1}{6g}$ if the reduction graph at $v$ is `elementary' in the sense that every edge is included in at most one cycle (the latter fact was already proved by Zhang in \cite{zh}). From (\ref{Zhangconj}) it is then clear that in particular the number $\varphi(X_v)$ is non-negative.  
Note that one might view our Corollary \ref{cor} as an archimedean analogue of Cinkir's result; the asymptotics moreover have similar shapes.

To finish this section we remark that by \cite{zh}, Section 1.4 there exists a natural non-archimedean analogue of the $\lambda$-invariant as well. By an application of the Noether formula for semistable arithmetic surfaces \cite{fa} \cite{mb}, equation (\ref{omegasquared}) translates into the formula
\begin{equation} \label{lambdaandheight}
  \deg \det \mathrm{R}\pi_* \omega = \frac{g-1}{6(2g+1)} \langle \Delta_\xi,
\Delta_\xi \rangle + \sum_v \lambda(X_v) \log Nv \, .
\end{equation}
It follows that the local $\lambda$-invariants serve to connect the self-intersection of the Gross-Schoen cycle with the (non-normalized) stable Faltings height $\deg \det \mathrm{R}\pi_* \omega$ of $X$ over $k$ (cf. \cite{zh}, equation (1.4.2)).

\section{Preliminaries} \label{prelim}

In this section we review some notions and results from the papers
\cite{hrgeom} and \cite{hrar} by Hain and Reed. We follow these sources quite closely, the most important difference being that we will usually work on the level of differential forms rather than on the level of cohomology classes.

As is customary, we view the moduli
spaces $\aa_g$ and $\mm_g$ of principally polarized complex
abelian varieties and of smooth
projective complex curves, respectively, as orbifolds.
Let $(V_\zz, Q : \wedge^2 V_\zz \to \zz(-n))$ be a polarized
integral Hodge structure of odd weight $n=-2i+1$ and let $\mathrm{GSp}_{2g} \to
\mathrm{GSp}(V_\zz,Q)$ be an algebraic representation, together with a lift of the structure morphism $\SS \to \mathrm{GSp}(V_\rr,Q)$, where $\SS$ is the Deligne torus, to $\mathrm{GSp}_{2g,\rr}$.
Let $(\vv_\zz,\QQ)$ be the corresponding variation of polarized Hodge
structures over $\aa_g$. We denote by
$\jj(V_\zz)$ the Griffiths intermediate jacobian fibration over $\aa_g$
associated to $\vv_\zz$. Thus, if $V_{A}$ is the fiber of the local system
$\vv_\zz$ at the point $A$ of $\aa_g$, the fiber of $\jj(V_\zz)$ at $A$ is the
complex torus $J(V_A)=(V_A \otimes \CC)/(F^{-i+1} (V_A \otimes \CC) +
\mathrm{Im} \, V_A)$. The holomorphic tangent bundle of $J(V_A)$ is equipped with
a canonical hermitian inner product derived from~$Q$.
This hermitian inner product determines a
translation-invariant global $2$-form on $J(V_A)$.
\begin{prop} \label{unique2form}
There exists a unique $2$-form $w_V$ on $\jj(V_\zz)$ such that the restriction of $w$
to each fiber over $\aa_g$ is the translation-invariant form associated to $Q$, and such that
the restriction of $w$ along the zero-section is trivial.
\end{prop}
\begin{proof} This is in \cite{hrar}, Section~5.
\end{proof}
We also mention the following result. 
Suppose that $V_\zz$ has weight~$-1$. From \cite{hain}, Section 3 we recall that the (standard $\gm$-) biextension line bundle $\mathcal{B}$ associated to $V_\zz$ is the set of isomorphism classes of mixed Hodge structures whose weight graded quotients are isomorphic to $\zz, V_\zz$ and $\zz(1)$. It has a natural projection to the product $J(V_\zz) \times \check{J(V_\zz)}$ where $J(V_\zz)= \mathrm{Ext}_{\mathcal{H}}(\zz,V_\zz)$ is the Griffiths intermediate jacobian of $V_\zz$, given by $M \mapsto (M/W_{-2}M,W_{-1}M)$. This projection equips $\mathcal{B}$ with the structure of a line bundle over $J(V_\zz) \times \check{J(V_\zz)}$. The polarization of $V_\zz$ furnishes a canonical morphism $\lambda \colon J(V_\zz) \to \check{J(V_\zz)}$. By pulling back along $(\mathrm{id},\lambda)$ one obtains from $\mathcal{B}$ a line bundle $\hat{\mathcal{B}}$ over $J(V_\zz)$. By abuse of language we refer to $\hat{\mathcal{B}}$ as the biextension line bundle over $J(V_\zz)$. Proposition 7.3 of \cite{hrar} then states the following.
\begin{prop} \label{biextension}
Suppose that $\vv_\zz$ is a variation of polarized
Hodge structures of weight~$-1$ over $\mathcal{A}_g$. 
Let $\hat{\mathcal{B}}$ be the 
biextension line bundle
over $\jj(V_\zz)$, obtained by applying the above construction to each of the fibers of $\jj(V_\zz)$. Then $\hat{\mathcal{B}}$ has a canonical
hermitian metric. The first Chern form of $\hat{\mathcal{B}}$ with this metric is equal to $2\,w_V$.
\end{prop}
We will be mainly concerned with the
cases where $V_\zz$ is equal to either $H$, $\wedge^3H$ or $\wedge^3H/H$,
where $H=H_1(X,\zz)$ is the first homology group of a compact Riemann
surface $X$ of genus $g \geq 2$.
The polarization is given by
the standard intersection form $Q_H=(,)$ on $H$.
Note that the form $Q_H$ identifies $H$ with its dual.
The Hodge structure $H$ is mapped into $\wedge^3H$ by sending
$x$ to $x \wedge \zeta $, where $\zeta$ in $\wedge^2 H$ is the dual of $Q_H$.

The polarizations on the Hodge structures $\wedge^3 H$ and
$\wedge^3H/H$ are given
explicitly as follows (cf. \cite{hrar}, p.~204). The form $Q_{\wedge^3H}$ on $\wedge^3H$ sends
\[ (x_1\wedge x_2 \wedge x_3, y_1 \wedge y_2 \wedge y_3 ) \mapsto \det(x_i,y_j) \, . \]
Next, one has a contraction map $c : \wedge^3 H \to H$, defined by
\begin{equation} \label{contraction} x \wedge y \wedge z \mapsto (x,y)z + (y,z)x + (z,x)y \, .
\end{equation}
One may verify that the composite $H \to \wedge^3 H \to H$ induced by $c$ and $\wedge \zeta$ is equal to $(g-1)$ times the
identity.
Denote the projection $\wedge^3 H \to \wedge^3H/H$ by $p$.
The projection $p$ has a canonical splitting $j$
(after tensoring with $\qq$), defined by
\[ p(x \wedge y \wedge z) \mapsto x \wedge y \wedge z - \zeta \wedge c(x\wedge y \wedge z)/(g-1) \, . \]
With these definitions, the form $Q_{\wedge^3H/H}$ on $\wedge^3H/H$
is given by
\[ (u,v) \mapsto (g-1) Q_{\wedge^3 H}(j(u),j(v)) \, . \]
We denote by $w_H$, $w_{\wedge^3 H}$ and $w_{\wedge^3H/H}$ the $2$-forms on the Griffiths intermediate
jacobian fibrations $\jj(H)$, $\jj(\wedge^3 H)$ and $\jj(\wedge^3H/H)$ over
$\aa_g$ whose existence is asserted
by Proposition~\ref{unique2form}. Note that
$\jj(H)$ is just the universal abelian variety over $\aa_g$.
\begin{prop} \label{equality2forms}
On $\jj(\wedge^3 H)$, the equality of $2$-forms
\[ (g-1)w_{\wedge^3 H} = c^* w_H + p^* w_{\wedge^3H/H} \]
holds.
\end{prop}
\begin{proof} According to \cite{hrgeom}, Proposition 18 we have $(g-1)Q_{\wedge^3 H}=c^*Q_H+p^*Q_{\wedge^3H/H}$. We obtain the result
by taking the associated canonical $2$-forms.
\end{proof}
Let $\pi \colon \cc_g \to \mm_g$ be the universal curve over $\mm_g$, viewed as
an orbifold.
As is explained in \cite{hrgeom}, Introduction we have a commutative diagram
\[ \xymatrix{ & \jj(H)   \\
\cc_g \ar[dd]_\pi \ar[ur]^\kappa \ar[r]^\mu \ar[dr]^\nu & \jj(\wedge^3 H)
\ar[u]^c \ar[d]^p
\\
 \ar[d] & \jj(\wedge^3H/H) \ar[d] \\
\mm_g \ar[r] & \aa_g  \, .}   \]
Here $\kappa$ is the map sending a pair $(X,x)$ where $X$ is a curve and $x$ is a point on $X$ to the
class of $(2g-2)x - \omega_X$ in the jacobian $J$ of $X$. The map $\mu$ is called the `pointed harmonic volume' (introduced by B. Harris, cf. \cite{harris}) and
sends a pair $(X,x)$ to the point associated, by the Griffiths Abel-Jacobi map, to the Ceresa cycle at $x$, i.e. the
(homologically trivial) cycle in $J$ given as $X_x-X^-_x$ where $X_x$ is the curve $X$ embedded in $J$ using $x$ and $X^-_x=[-1]_*X_x$.
The map $\nu$ is called the `harmonic volume' and is just defined as
the composite of $\mu$ with the map $p \colon \jj(\wedge^3H) \to
\jj(\wedge^3H/H)$ induced by the projection $\wedge^3H \to \wedge^3 H/H$. The map $\nu$ factors over $\mm_g$,
hence defines a Griffiths normal function $\mm_g \to \jj(\wedge^3H/H)$ that we shall also denote by $\nu$.

It will be useful to pass from $\mm_g$, $\cc_g$ and $\jj(H)$ to the level-2 moduli orbifolds $\mm_g[2]$, $\cc_g[2]$ and $\jj(H)[2]$; see for example
\cite{halo}, Section~7.4 for precise definitions. The orbifold $\mm_g[2]$ can be endowed with a universal theta characteristic $\alpha$,
i.e. a consistent choice of an element $\alpha \in \mathrm{Pic}^{g-1} X$ for each curve $X$
such that $2\alpha$ is the canonical divisor class. We consider the map
\[ j_\alpha \colon \cc_g[2] \longrightarrow \jj(H)[2] \]
given by sending $(X,x)$ to the class of $(g-1)x-\alpha$ on the jacobian $J$ of $X$. Note that $\kappa = 2j_\alpha$.  

Let $e^J$ be the $2$-form
\begin{equation} \label{eJ} e^J = -\frac{1}{2g(2g+1)}(2 \, \kappa^* w_H+ 3 \, \mu^* w_{\wedge^3H})
\end{equation}
over $\cc_g$. By a result of Morita \cite{molinear} (see also \cite{hrgeom}, Theorem
6) this $2$-form represents the class of
$\omega^{-1}_{\cc_g/\mm_g}$ in $H^2(\cc_g,\qq)$, where $\omega_{\cc_g/\mm_g}$
is the relative dualizing sheaf of $\cc_g$ over $\mm_g$. Recall from the Introduction that
we have a $2$-form $\omega_{\mathrm{HR}}$ on $\mm_g$ by taking the pullback, along $\nu$,
of the first Chern form of $(\hat{\mathcal{B}},\|\cdot\|_{\hat{\mathcal{B}}})$,
and dividing by $8g+4$.
\begin{prop} \label{firstequality}
Over $\cc_g[2]$, we have an equality
\[ j_\alpha^* w_H = -\frac{g(g-1)}{2} e^J - \frac{3}{2} \omega_{\mathrm{HR}} \]
of $2$-forms.
\end{prop}
\begin{proof}
Upon replacing $H_1(X,\zz)$ by $H_1(X,\zz(-1))$ one views
the variation of Hodge structures over $\aa_g$ determined by $\wedge^3 H/H$ to be one of weight~$-1$ (cf.
\cite{hrar}, Section~4). Proposition \ref{biextension} gives
that the first Chern form of $(\hat{\mathcal{B}},\|\cdot\|_{\hat{\mathcal{B}}})$
equals $2 \, w_{\wedge^3H/H}$ so that
\[  \nu^*w_{\wedge^3H/H} =(4g+2)\,\omega_{\mathrm{HR}} \, . \]
Proposition \ref{equality2forms} then yields
\[ (g-1)\mu^* w_{\wedge^3H}=\kappa^* w_H + (4g+2)\,\omega_{\mathrm{HR}} \, . \]
Combining this equality with the definition of $e^J$ we find
\[ \kappa^*w_H = -2g(g-1)e^J - 6 \, \omega_{\mathrm{HR}}  \]
(cf. \cite{hrgeom}, Theorem~1). On the other hand we have $[2]^* w_H = 4w_H$, which follows from the fact that $w_H$ restricts to a translation-invariant $(1,1)$-form in each fiber, and $\kappa = 2j_\alpha$ which together give
\[ \kappa^*w_H = 4 j_\alpha^* (w_H) \, . \]
The proposition follows.
\end{proof}

\section{Kawazumi's result}

In this section we state Kawazumi's result on the second variation of the $\varphi$-invariant on $\mm_g$. His result expresses the second variation of $\varphi$ in terms of the differential form $e^J$ on $\cc_g$, introduced above, and a second differential form $e^A$ on $\cc_g$ which we introduce next.

Let $X$ be a compact Riemann surface of genus $g \geq 2$. From \cite{ar} we
obtain that the line bundle $\mathcal{O}(\Delta)$ on $X \times X$, where $\Delta$
is the diagonal, comes equipped with a natural hermitian metric given by
$\|1\|(x,y) = G(x,y)$, where $G$ is the Arakelov Green's function. Fixing $x$ on $X$, the function $G(x,\cdot)$ is determined by the set of equations
\[  \frac{\partial \overline{\partial} }{\pi i} \log G(x,\cdot) = \mu_X - \delta_x \, , \qquad \int_X \log G(x,y) \, \mu_X(y) = 0 \, . \]
By demanding that the adjunction (residue) isomorphism
\[ \mathcal{O}(-\Delta)|_\Delta \rightarrow \omega_X \]
where $\omega_X$ is the holomorphic cotangent bundle to $X$
should be an isometry we obtain a canonical hermitian metric $\|\cdot\|_\mathrm{Ar}$
on $\omega_X$. Globalizing this construction we obtain a canonical hermitian metric
$\|\cdot\|_\mathrm{Ar}$ on $\omega_{\cc_g/\mm_g}$. Denote by $e^A$ the first
Chern form of the dual metric on $\omega^{-1}_{\cc_g/\mm_g}$. Kawazumi's theorem is then the following.
\begin{thm} (Kawazumi \cite{kaw}) \label{kawazumi}
On $\cc_g$, the differential equation
\[ e^A - e^J = \frac{1}{2g(2g+1)} \frac{\deldelbar }{\pi i} \, \varphi \]
is satisfied.
\end{thm}
Actually the main result of \cite{kaw} reads
\[ e^A - e^J = \frac{-2i}{2g(2g+1)} \deldelbar \, a_g \, . \]
One verifies directly that the function $a_g$ as defined in the Introduction
of \cite{kaw} is equal to $\frac{1}{2\pi} \varphi$, and that the $2$-form $e^A$ on $\cc_g$ as defined in \cite{kaw}
is the one defined above. We would like to explain that the $2$-form $e^J$ defined in (\ref{eJ}) is equal to the $2$-form called $e^J$ in \cite{kaw}. The latter is written (cf. Definition (3.10) in \cite{kaw}) as
\[ e^J = -\frac{1}{2g(2g+1)} (M_1 + M_2)(\eta_1^{\otimes 2}) \, , \]
where the following notation is used.
Let $\hh_\zz$ be the local system over $\aa_g$ associated to $H$ and consider
the derived local systems $\hh_\rr=\hh_\zz
\otimes \rr$ and $\hh_\CC = \hh_\zz \otimes \CC$ over $\aa_g$ and $\mm_g$. We use the same notation to denote
their pullbacks on $\cc_g$. Note that when pulled back along $\cc_g$, the
intermediate jacobian fibration $\jj(\wedge^3 H)$ can be seen as a torus bundle
over $\cc_g$ with fiber $\wedge^3H \otimes (\rr/\zz)$.
Both $M_1,M_2$ are real forms in  $\mathrm{Hom}(\wedge^2(\wedge^3 \hh_\CC),
\CC)$,
hence global $2$-forms on $\jj(\wedge^3 H)$, coinciding with the forms $C_1,C_2$ from \cite{molinear}. By the
discussion in Remark~20 of \cite{hrgeom} we can therefore write $M_1=2c^* w_H$ and $M_2=3w_{\wedge^3 H}$ on
$\jj(\wedge^3 H)$ where $c \colon
\wedge^3 H \to H$ is the contraction map (\ref{contraction}). The section
$\eta_1^{\otimes 2}$ of the local system
$\wedge^2(\wedge^3 \hh_\CC)$ over $\cc_g$
is the one induced by the section $\eta_1 = \eta'_1 + \overline{\eta}'_1$ of the
local system $\wedge^3 \hh_\CC$ where,
as is explained in the introduction to \cite{kaw}, the section $\eta'_1$ of
$\wedge^3 \hh_\rr$ is the first
variation of the pointed harmonic volume
$\mu \colon \cc_g \to \jj(\wedge^3 H)$.
We obtain $M_1(\eta_1^{\otimes 2})=2 \mu^* c^* w_H=2 \kappa^* w_H$ and
$M_2(\eta_1^{\otimes 2})= 3\mu^* w_{\wedge^3 H}$ and the equality of Kawazumi's $e^J$ with the one in (\ref{eJ}) follows.

\section{Proof of the main theorem}

In this section we prove Theorems \ref{realmain} and \ref{main}. Let $\delta_F$ be the Faltings delta-invariant on $\mm_g$ (see \cite{fa}, p.~402 for its definition). A convenient expression for its second variation is given by the following proposition.
\begin{prop} \label{secondequality}
Over $\cc_g[2]$, we have an equality
\[ j_\alpha^* w_H = -\frac{g(g-1)}{2}e^A - \frac{3}{2}\omega_{\mathrm{Hdg}} - \frac{1}{8}
\frac{\deldelbar}{\pi i} \, \delta_F \]
of $2$-forms.
\end{prop}
\begin{proof} We refer to \cite{fa}, p.~413 for the first half of this proof.
On $\mathcal{J}(H)[2]$ we have a universal theta divisor
$\Theta_\alpha$. When restricted to the jacobian $J$ of a curve $X$, the divisor
$\Theta_\alpha$ is equal to the image of the canonical theta divisor on
$\mathrm{Pic}^{g-1} X$ under the isomorphism $\mathrm{Pic}^{g-1} X \to J$ defined
by $x \mapsto x-\alpha$. Further, the orbifold $\mathcal{J}(H)[2]$ can be
written as a
quotient of the analytic variety $\CC^g \times \HH_g$ where $\HH_g$ is the
Siegel upper half space of complex symmetric $g$-by-$g$ matrices with positive definite
imaginary part. When pulled back to $\CC^g \times \HH_g$, for a suitable choice
of universal theta characteristic
the divisor $\Theta_\alpha$ can be given analytically by Riemann's
standard theta function $\theta$. As a result,
the line bundle $\mathcal{O}(\Theta_\alpha)$ on $\jj(H)[2]$
comes equipped with a natural hermitian metric;
the norm of $\theta$ in this metric is given by
\[ \|\theta\| =(\det \mathrm{Im} \, \tau)^{1/4}\exp(-\pi \, {}^t
y \, (\mathrm{Im} \, \tau)^{-1} \, y)|\theta(z,\tau)| \]
where $z=x+iy $ is in $\CC^g$ and $\tau$ is in $\HH_g$. With this metric,
the first Chern form $w_0$ of $\mathcal{O}(\Theta_\alpha)$ equals
\begin{equation} \label{w_0_H} w_0 = w_H + \frac{1}{2}
\omega_{\mathrm{Hdg}}
\end{equation}
(cf. \cite{hrgeom}, Proposition 2).
Now as is explained in \cite{hrgeom}, Section 3
there exists a canonical isomorphism
\[ j_\alpha^* \mathcal{O}(\Theta_\alpha) \longrightarrow
\omega^{\otimes g(g-1)/2} \otimes \mathcal{L}^{-1}   \]
of line bundles over $\cc_g[2]$,
given by sending $j_\alpha^* \theta$ to  a suitable Wronskian
differential. By Lemma 3.2 of \cite{dj}
the norm of this isomorphism is equal to $\exp(\delta_F/8)$,
when $\mathcal{L}$ is equipped with the Hodge
metric given by (\ref{defhodgemetric}), and $\omega$ is equipped with the
Arakelov metric $\| \cdot \|_{\mathrm{Ar}}$.
By taking first Chern forms we find
\[ j_\alpha^* w_0 = -\frac{g(g-1)}{2} e^A - \omega_{\mathrm{Hdg}} -
\frac{1}{8} \frac{\deldelbar}{\pi i} \, \delta_F \, . \]
We obtain the proposition by inserting (\ref{w_0_H}).
\end{proof}
From Propositions \ref{firstequality} and \ref{secondequality} we infer that
\begin{equation} \label{eAeJ}
 \frac{g(g-1)}{2} \left( e^A - e^J \right) = -\frac{1}{8} \frac{\deldelbar}{\pi
 i} \, \delta_F
+ \frac{3}{2} \omega_{\mathrm{HR}} - \frac{3}{2} \omega_{\mathrm{Hdg}} \, .
\end{equation}
By combining this equation with Kawazumi's result Theorem \ref{kawazumi} and
equation (\ref{lambdaphidelta}) we obtain Theorem \ref{main}.
\begin{proof}[Proof of Theorem \ref{realmain}] Theorem \ref{realmain} follows from Theorem \ref{main} once one knows that the only pluriharmonic functions on $\mm_g$ are constants. But this follows from the fact that $\mm_g$ allows a surjection from Teichm\"uller space in genus $g$, which is contractible, and the fact that the only invertible holomorphic functions on $\mm_g$ are constants (cf. \cite{hrar}, Lemma 2.1).
\end{proof}
\begin{remark} A shorter and perhaps more natural proof of Theorem \ref{realmain} (and hence of its corollaries) would be possible once one knows how to carry through some of the arguments in \cite{zh} in terms of line bundles on a suitable level cover $\mm'_g$ of $\mm_g$. For example, one would like to interpret Zhang's result (\ref{omegasquared}) as stating, among other things, that there exists a line bundle $\langle \Delta_\xi, \Delta_\xi \rangle$ on $\mm'_g$, together with a canonical isomorphism $\langle \Delta_\xi, \Delta_\xi \rangle^{\otimes 2g-2} \to \langle \omega, \omega \rangle^{\otimes 2g+1}$ of norm $\exp(-(2g-2)\varphi)$ over $\mm'_g$, where $\langle \omega,\omega \rangle$ is Deligne's pairing of the relative dualizing sheaf $\omega$ with itself. We will return to these matters in a future paper.
\end{remark}

\section{Hyperelliptic curves} \label{hyperelliptic}

In this section we prove Theorem \ref{hyp}. Let $\hh_g$ be the orbifold
moduli space of complex hyperelliptic
curves of genus $g \geq 2$. We start by reviewing the construction of the discriminant modular form $\Delta_g$ on $\hh_g$.  It generalizes the usual discriminant modular form of weight~$12$ in the theory of moduli of elliptic curves. 

Let $n={2g \choose g+1}$ and $r={2g+1 \choose g+1}$. Let $\HH_g$ again be the Siegel upper half-space of symmetric complex $g \times g$-matrices with
positive definite imaginary part. For $z$ in  $\CC^g$ (viewed as a
column vector), a matrix $\tau$ in $\HH_g$ and $\eta,\eta'$ in 
$\frac{1}{2} \zz^g$ we define the (classical) theta function with characteristic
$\eta=[{\eta' \atop \eta''}]$ to be 
\[ \theta[\eta](z,\tau) = \sum_{n \in \zz^g} \exp( \pi i  {}^t (n+\eta')
\tau  (n+\eta') + 2\pi i  {}^t (n+\eta') (z+\eta'')) \, . \] For
any subset $S$ of $\{ 1,2,\ldots,2g+1\}$ we define a theta
characteristic $\eta_S$ as follows: let
\[ \begin{array}{rcl} \eta_{2k-1} & = &
\left[ { {}^t( 0 \, , \, \ldots \, , \, 0 \, , \, \frac{1}{2} \, ,
\,  0 \, , \, \ldots \, , \, 0 ) \atop
{}^t (\frac{1}{2} \, , \, \ldots \, , \,  \frac{1}{2} \, , \,  0 \,
, \, 0 \, , \,
\ldots \, , \, 0 )}
\right] \, , \quad 1 \leq k \leq g+1 \, ,  \\ \eta_{2k} & = &
\left[ { {}^t( 0 \, , \, \ldots \, , \, 0 \, , \, \frac{1}{2} \, ,
\,  0 \, , \, 
\ldots \, , \, 0 ) \atop
{}^t (\frac{1}{2} \, , \, \ldots \, , \, \frac{1}{2} \, , \, \frac{1}{2} \, ,
\, 0 \, , \, \ldots \, , \, 0
)} \right] \, , \quad 1 \leq k \leq g \, ,
\end{array} \] where each time the non-zero entry in the top row occurs in
the $k$-th position. Then we put $\eta_S = \sum_{k \in S} \eta_k$
where the sum is taken modulo 1.
Let $\tt$ be the set of subsets of $\{1,2,
\ldots, 2g+~1 \}$ of cardinality $g+1$. Write
$U=\{1,3,\ldots,2g+1\}$ and let $\circ$ denote the symmetric
difference. The discriminant modular form $\Delta_g$ is then defined to be the function
\[ \Delta_g(\tau) = 2^{-(4g+4)n} \prod_{T \in \tt}
\theta[\eta_{T \circ U}](0,\tau)^8   \] on $\HH_g$. It follows from results in
\cite{lock}, Section~3 that the
function $\Delta_g$ is a modular form on the congruence subgroup $\Gamma_g(2) = \{ \gamma \in \mathrm{Sp}(2g,\zz) | \gamma \equiv I_{2g} \bmod 2 \}$ of
weight $4r$.

Now let $\tau$ in $\HH_g$ be the period matrix of a complex hyperelliptic curve $X$ of genus $g$ marked with a \emph{canonical} basis of homology determined by an ordering of the set of Weierstrass points on $X$ (see \cite{mu}, Chapter IIIa, \S 5). We put
$\|\Delta_g\|(\tau) = (\det \mathrm{Im} \, \tau)^{2r}|\Delta_g(\tau)|$. Then
for a given hyperelliptic curve $[X] \in \hh_g$ the value of $\|\Delta_g\|(\tau)$ on a period matrix on a canonical basis associated to $X$ does not depend on the choice of such a matrix.  We find that $\|\Delta_g\|$ is a well-defined real-valued function on $\hh_g$.

We remark that $\hh_g$ extends as a moduli stack of hyperelliptic curves over
$\mathbb{Z}$. Further, there exists an up to sign unique global trivializing section
$\Lambda$ of the line bundle $\mathcal{L}^{\otimes 8g+4}$ over $\hh_g$ that extends as a trivializing section
of $\mathcal{L}^{\otimes 8g+4}$ over $\mathbb{Z}$ (cf. \cite{djexplicit}, Proposition~3.1). It is possible to give an explicit formula for $\|\Lambda\|_{\mathrm{Hdg}}$ over $\hh_g$ in terms of $\|\Delta_g\|$.
\begin{prop} \label{formulanormlambda} Let $\Lambda$ be the (up to sign unique) global trivializing section of $\mathcal{L}^{\otimes 8g+4}$ over $\zz$. Then the formula 
\[ \|\Lambda\|_{\mathrm{Hdg}}^n =
(2\pi)^{4g^2r}\|\Delta_g\|^g  \]
holds.
\end{prop}
\begin{proof} For this we refer to the proof of Theorem~8.2 in \cite{djexplicit}.
\end{proof}
For the biextension metric on $\hh_g$ we have the following result (cf. \cite{hrar}, Proposition 6.7).
\begin{prop} The metric $\|\cdot\|_\mathrm{biext}$ restricted to the trivial line bundle $\mathcal{L}^{\otimes 8g+4}$ over $\hh_g$ is a constant metric.
\end{prop}
\begin{proof} The Ceresa cycle is zero for any hyperelliptic curve $X$. Indeed, when
$X$ is embedded into its jacobian using a Weierstrass point, the involution $[-1]$ on the jacobian restricts to the hyperelliptic involution on $X$. Further, at zero the biextension line bundle $\hat{\mathcal{B}}$ restricts canonically to $\CC$ with its standard euclidean metric. 
\end{proof}
It follows that $\beta=-\log\|\Lambda\|_{\mathrm{Hdg}}$, up to a constant. From Theorem \ref{realmain} we then deduce
\[ (8g+4)\lambda=-\log\|\Lambda\|_{\mathrm{Hdg}} \, , \] 
up to a constant. Upon applying Proposition \ref{formulanormlambda} one then obtains 
\[ (8g+4)n \, \lambda = -4g^2r\log(2\pi) - g \log \|\Delta_g \|= 
-(8g+4)ng\log(2\pi) - g \log \|\Delta_g \| \, , \]
up to a constant, and Theorem \ref{hyp} is proven.

Using a recent result of K. Yamaki \cite{ya} it is possible to actually compute the constant implied by Theorem \ref{hyp}. Let $X$ be a hyperelliptic curve of genus $g \geq 2$ with semi-stable reduction over a non-archimedean
local field $k$. Let $\varepsilon$ be Zhang's epsilon-invariant of $X$ (cf. \cite{zh}, Section~1.2). Define the invariant $\psi$ as
\[ \psi = \varepsilon + \frac{2g-2}{2g+1}\varphi \, . \]
Let $\mathcal{X}$ be the special fiber of a regular semistable model of $X$ over the
ring of integers of $k$. We say that a double point $x$ of $\mathcal{X}$ is of
type $0$ if the local normalization of $\mathcal{X}$ at $x$ is connected. We
say that $x$ is
of type $i$, where $i=1,\ldots,[g/2]$, if the local normalization of
$\mathcal{X}$ at $x$ is the disjoint union of a curve of  genus
$i$ and a curve of genus $g-i$. Let $\iota$ be the
involution on $\mathcal{X}$ induced by the hyperelliptic involution on $X$.
Let $x$ be a double point of type $0$ on $\mathcal{X}$. If $x$ is fixed by
$\iota$, we say that $x$ is of subtype $0$. If $x$ is not fixed by $\iota$, the
local normalization of $\mathcal{X}$ at $\{x,\iota(x)\}$ consists of two
connected components, of genus $j$ and $g-j-1$, say, where $1\leq j \leq
[(g-1)/2]$. In this case we say that
the pair $\{x,\iota(x)\}$ is of subtype $j$. Let $\xi_0$ be the number of double
points of subtype $0$, let $\xi_j$ for $j=1,\ldots,[(g-1)/2]$ be the number of
pairs of double points of subtype $j$, and let $\delta_i$ for $i=1,\ldots,[g/2]$
be the number of double points of type $i$.
Equality (1.2.5) and Theorem 3.5 of \cite{ya} imply that
\[ \psi = \frac{g-1}{2g+1}\xi_0 + \sum_{j=1}^{[(g-1)/2]}
\frac{6j(g-j-1)+2g-2}{2g+1} \xi_j +
\sum_{i=1}^{[g/2]} \left( \frac{12i(g-i)}{2g+1} -1 \right) \delta_i  \, . \]
By \cite{zh}, Section~1.4 the non-archimedean $\lambda$-invariant is given by 
\[  \lambda = \frac{g-1}{6(2g+1)} \varphi + \frac{1}{12}(\varepsilon + \delta) =
\frac{1}{12}(\psi + \delta) \, .  \]
Here $\delta$ denotes the total number of singular points in the fiber at $v$. We obtain 
\[ (8g+4)\lambda = g\xi_0 + \sum_{j=1}^{[(g-1)/2]} 2(j+1)(g-j)\xi_j + \sum_{i=1}^{[g/2]} 4i(g-i)\delta_i \, .  \]
By the local Cornalba-Harris equality \cite{ch} \cite{ka} \cite{yacorn} this simplifies to
\[ (8g+4)\lambda = -\log \|\Lambda\| \]
where now the right hand side denotes the order of vanishing of $\Lambda$ along the closed point of the spectrum of the ring of integers of $k$. Now take a hyperelliptic curve $X$ of genus $g$ over $\qq$. As $\Lambda$ furnishes a non-zero section of the Hodge bundle $\det \mathrm{R}\pi_*\omega$ we have the formula
\[ (8g+4) \deg \det \mathrm{R}\pi_* \omega = - \sum_v \log \| \Lambda \|_v \log Nv \]
for the (non-normalized) stable Faltings height of $X$ over a finite field extension of $\qq$ where $X$ acquires semi-stable reduction. By equation (\ref{lambdaandheight}) and the known vanishing of $\langle \Delta_\xi,\Delta_\xi \rangle$ in the hyperelliptic case (cf. Section~4 of \cite{gs}),
one obtains that the constant implied by Theorem \ref{hyp}   actually vanishes. \\
 
\noindent \textbf{Acknowledgements} The author would like to thank B. Edixhoven, N. Kawazumi, B. Moonen and the referee for several helpful remarks.

\vspace{0.5cm}

\noindent Address of the author: 
Mathematical Institute,
University of Leiden,
PO Box 9512,
2300 RA Leiden,
The Netherlands. \\
Email: \verb+rdejong@math.leidenuniv.nl+

\end{document}